\documentclass[12pt, leqno]{amsart}
\usepackage{amsmath}
\usepackage{mathtools}
\usepackage{amssymb}
\usepackage{amsthm}
\usepackage{graphicx}
\usepackage{enumerate}
\usepackage[mathscr]{eucal}
\usepackage{tikz}
\usepackage{mlmodern}
\usetikzlibrary{decorations.text,calc,arrows.meta}
\theoremstyle{plain}

\usepackage[margin=0.9in]{geometry}
\numberwithin{equation}{section}
\usepackage[english]{babel}

\newtheorem{theorem}{Theorem}[section]
\newtheorem{cor}[theorem]{Corollary}
\newtheorem{Lemma}[theorem]{Lemma}
\theoremstyle{definition}
\newtheorem*{acknowledgement}{\textnormal{\textbf{Acknowledgements}}}
\newtheorem{definition}[theorem]{Definition}
\newtheorem{remark}{Remark}[theorem]

\usepackage[pagewise]{lineno}
\usepackage{hyperref}
\hypersetup{
    colorlinks=false,
    linkcolor=blue,
    filecolor=magenta,      
    urlcolor=cyan,
}
\newcommand{\x}{\mathbf{x}}

\newcommand{\z}{\mathbf{z}}
\newcommand{\T}{\mathsf{T}}
\newcommand{\A}{\mathsf{A}}
\newcommand{\X}{\mathbb{X}}

\newcommand{\Y}{\mathbb{Y}}
\newcommand{\M}{\mathscr{M}}
\newcommand{\B}{\mathcal{B}}
\newcommand{\F}{\mathbb{F}}
\newcommand{\C}{\mathbf{Co}}
\newcommand{\R}{\mathfrak{Re}}
\newcommand{\HS}{\mathbb{H}}
\begin{document}

\title[The Weak differentiability of norm]{The Weak differentiability of norm and a generalized Bhatia-\v{S}emrl Theorem}

\author[Saikat Roy]{Saikat Roy}

\address[Roy]{Department of Mathematics\\ Indian Institute of Technology Bombay\\ Mumbai \\ India}
\email{saikatroy.cu@gmail.com}

                \newcommand{\acr}{\newline\indent}


\thanks{The research of the author is sponsored by the Institute Post Doctoral Fellowship at Indian Institute of Technology Bombay.}

\begin{abstract}
We completely characterize the weak differentiability (or, in other words Gateaux differentiability) of the norm in the spaces of bounded multilinear maps. Also, we obtain a multilinear generalization of the well-known Bhatia-\v{S}emrl theorem on Birkhoff-James orthogonality.
\end{abstract}

\maketitle

\section{Introduction}
The aim of this article is to characterize the weak differentiability of norm in the spaces of bounded multilinear maps that works equally well in the contexts of (real and complex) normed linear spaces and Hilbert spaces, regardless of the dimension (finite or infinite).
Let $(\X, \|\cdot\|)$ be a normed linear space over a field $\mathbb{F}$ ($\mathbb{R}$ or $\mathbb{C}$). The norm $\|\cdot\|$ is said to be \emph{weakly differentiable} or \emph{Gateaux differentiable} at a non-zero point $x_0\in \X$, if 
\[\lim_{t\to 0} \frac{\|x_0+ty\|-\|x_0\|}{t}\]
exists finitely for all $y\in \X$. The study of differentiability of norm in the spaces of operators has a rich literature background \cite{Abatzoglou, Heinrich, Hennefeld, Holub, Keckic}. The initiation of such a study can be traced back to the seminal works due to Heinrich, followed by Abatzoglou and Hennefeld. More recently, the weak differentiability of norm in the spaces of bounded linear (bilinear) operators have been studied in \cite{PSG, Sain2, SPMR}, in the setting of real normed linear spaces. It is worth mentioning that the current work presents a tractable characterization of Birkhoff-James orthogonality in the spaces of multilinear maps which facilitates a significant improvement of the previously obtained results on the weak differentiability of the norm in the spaces of linear (bilinear) operators.

\bigskip

\textbf{Notations and terminologies}. Throughout the paper, $\X$, $\X_i,~\Y$ denote normed linear spaces ($1\leq i \leq k~;$ $k\in \mathbb{N}$), over the field $\F$ of real or complex numbers ($\mathbb{R}$ or $\mathbb{C}$). The symbol $\HS$ is reserved for a Hilbert space equipped with the standard inner product $\langle \cdot, \cdot \rangle.$ A \emph{$k$-linear map} $\T:\X_1\times \dots \times \X_k \to \Y$ is a function that is separately linear in each variable. For simplicity, we denote the norm in each $\X_i$ ($1\leq i \leq k$) and $\Y$ by the same symbol $\|\cdot\|.$ The map $\T$ is said to be \emph{bounded}, if there exists a positive constant $C$ such that $\|\T(x_1, \dots, x_k)\|\leq C\|x_1\|\dots\|x_k\|$ for all $(x_1,\dots,x_k)\in \X_1\times \dots \times \X_k.$ In that case the norm of $\T$ is defined by 
\[\|\T\|=\sup\{\|\T(x_1,\dots, x_k)\|:~\|x_1\|=\dots=\|x_k\|=1\}.\]
The product space $\X_1\times \dots \times \X_k$ is always equipped with the product topology. Note that a bounded $k$-linear map is continuous and vice versa. We use the symbol $\x$ to denote a member $(x_1,\dots, x_k)$ in $\X_1\times \dots \times \X_k$ while the members of a normed space $\X$ are simply written as $x.$ The zero vector in any space other than the scalar field is written as $\theta$. The closed unit ball and the unit sphere of a normed linear space $\X$ are denoted by $B_\X$ and $S_\X$, respectively. The symbol $\X^*$ stands for the topological dual of $\X.$ We write $\B^{k}(\X_1\times \dots \times \X_k, \Y)$ to denote the space of all bounded $k$-linear maps from $\X_1\times \dots \times \X_k$ to $\Y$, endowed with the usual supremum norm on $S_{\X_1}\times \dots \times S_{\X_k}.$ In particular when $k=1$, the space $\B^{1}(\X,\Y)$ (or $\B^{1}(\X)$ when $\X=\Y$) is the space of all bounded linear operators from $\X$ to $\Y$, endowed with the usual operator norm. The identity operator on $\X$ is denoted by $\mathbf{I}.$ Let $M_\T$ stand for the norm attainment set of the $k$-linear map $\T$, i.e.,
\[M_\T:=\left\{\x\in S_{\X_1}\times \dots \times S_{\X_k}:~\|\T \x\|= \|\T\|\right\}.\]
\noindent A standard compactness argument ensures that $M_\T\neq \emptyset$, when $\X_i$ is finite-dimensional for each $i=1,\dots,k$. We cannot ensure the same if any $\X_i$ is infinite-dimensional. However, in that case we can always ensure the existence of a norming sequence $(\x_n) \subseteq S_{\X_1}\times \dots \times S_{\X_k}$ such that $\|\T\x_n\|\to \|\T\|$. By Hahn-Banach Theorem, for each $n\in \mathbb{N}$, we can find $y_n^*\in S_{\Y^*}$ such that $y_n^*(\T\x_n)=\|\T\x_n\|.$ Thus, the following collection of sequences
\[\M_\T:=\left\{((\x_n, y_n^*))\subseteq \left(S_{\X_1}\times \dots \times S_{\X_k}\right)\times S_{\Y^*}:~y_n^*(\T \x_n)\to \|\T\|\right\},\]
\noindent is always non-empty, regardless $\X_i,\Y$ are finite or infinite-dimensional for $1\leq i \leq k.$

\bigskip

Following Birkhoff \cite{Birkhoff}, we say that an element $x\in \X$ is Birkhoff-James orthogonal to another element $y\in \X$, written as $x\perp_B y$, if $\|x+\lambda y\|\geq \|x\|$ for all scalars $\lambda.$ James \cite{James2} showed that $x\perp_B y$ if and only if there exists $x^*\in S_{\X^*}$ such that $x^*(x)=\|x\|$ and $x^*(y)=0.$ Given any non-zero $x\in \X$, the non-empty collection $J(x):=\{x^*\in S_{\X^*}:~x^*(x)=\|x\|\}$ is known as the collection of all support functionals at $x$. The point $x$ is called smooth if $J(x)$ is a singleton set. Weak differentiability of norm can also be characterized in terms of Birkhoff-James orthogonality and support functionals. Let us recall that famous result due to James which will play the central role in this paper.

\begin{theorem}[\cite{James2}, Theorem 4.2]\label{Primary Result}
Let $\X$ be a normed linear space and let $x\in \X$ be non-zero. Then the following are equivalent:

\begin{itemize}
    \item[(i)] Birkhoff-James orthogonality is right-additive at $x$, i.e., $x\perp_B y$ and $x\perp_B z$ imply that $x\perp_B (y+z)$ for all $y, z\in \X$.
    
    \item[(ii)] $x$ is smooth.
    
    \item[(iii)] The norm in $\X$ is weakly differentiable at $x$.
\end{itemize}
\end{theorem}

\bigskip

Including the introductory part, the present work is demarcated into three sections. In the second section, we completely characterize the Birkhoff-James orthogonality of multilinear maps, in the setting of normed linear spaces, which is a new addition to the existing literature. The said characterization, whenever restricted to the setting of Hilbert space, leads to a multilinear generalization of the well-known Bhatia-\v{S}emrl Theorem in both finite and infinite-dimensional cases. Finally, in the third section, we employ the results of the second section to characterize the weak differentiability of the norm in the spaces of multilinear maps. The cumulative results of this section extend and strengthen some earlier results related to the weak differentiability of the norm in the spaces of linear operators in both finite and infinite-dimensional contexts \cite{PSG, Sain2, SPM}.

\vspace{1mm}  

\section{Orthogonality of Multilinear Maps}

\vspace{2mm}

We begin with a lemma which will be used in the sequel. Given any $\lambda \in \mathbb{C}$, let $\R~\lambda$ and $\overline{\lambda}$ denote the real part of $\lambda$ and conjugate of $\lambda$, respectively.

\begin{Lemma}\label{Lemma:1}
Let $\X_i,~\Y$ be normed linear spaces; $1\leq i \leq k,$ and let $\T,~\A \in \B^{k}\left(\X_1\times \dots \times \X_k, \Y\right)$. Then the collection
\begin{align}\label{Set Omega}
\Omega(\T,\A):=\left\{\lambda\in \F:~ y_n^*(\A \x_n)\to \lambda,~((\x_n, y_n^*))\in \M_\T\right\}, 
\end{align}
is a compact subset of $\F.$
\end{Lemma}

\begin{proof}
Evidently $\Omega(\T,\A)$ is bounded. Consider a sequence $(\lambda_j)\subseteq \Omega(\T,\A)$ such that $\lambda_j\to \lambda_0$, for some $\lambda_0\in \F$. We show that $\lambda_0\in \Omega(\T,\A).$ Clearly, for each $j\in \mathbb{N}$, there exists a sequence $\left(\x_{n,j}, y^*_{n,j}\right)\subseteq \left(S_{\X_1}\times \dots \times S_{\X_k}\right)\times S_{\Y^*}$ such that
\[{y^*_{n,j}\left(\A \x_{n,j}\right)}\to \lambda_j \quad \mathrm{and} \quad y^*_{n,j}\left(\T \x_{n,j}\right)\to \|\T\|,~\mathrm{as}~n\to \infty.\]
For each $j\in \mathbb{N}$, choose $m_j\in \mathbb{N}$ in such a way that the following hold:
\[(i)~\left|{y^*_{m_j,j}\left(\A \x_{m_j,j}\right)}-\lambda_j\right| < \frac{1}{j},\quad (ii)~\left|y^*_{m_j,j}\left(\T \x_{m_j,j}\right)-\|\T\|\right|< \frac{1}{j}.\]
\noindent For each $j\in \mathbb{N}$, put
\[\widetilde{\x}_j=\x_{m_j,j},~ \widetilde{y}^*_j=y^*_{m_j,j}\]
Note that
\begin{align*}
\left|{\widetilde{y}^*_j(\A \widetilde{\x}_j)}-\lambda_0\right|  & \leq \left|{\widetilde{y}^*_j(\A \widetilde{\x}_j)}-\lambda_j\right|+|\lambda_j-\lambda_0|\\
& < \dfrac{1}{j}+|\lambda_j-\lambda_0|.
\end{align*}

\noindent Letting $j\to \infty,$ we have ${\widetilde{y}^*_j(\A \widetilde{\x}_j)}\to \lambda_0.$ Also, $\widetilde{y}^*_j(\T\widetilde{\x}_j)\to \|\T\|$, as $j\to \infty$. Therefore, $((\widetilde{\x}_j, \widetilde{y}^*_j))\in \M_\T$ and $\lambda_0\in \Omega(\T,\A)$.
\end{proof}

For a non-empty subset $D$ of $\F$, let $\C(D)$ denote the convex hull of $D$. Note that $\C(D)$ is a compact subset of $\F$, whenever $D$ is compact. Therefore, the preceding lemma ensures that $\C(\Omega(\T,\A))$ is a compact subset of $\F$, for any $\T,~\A \in \B^{k}\left(\X_1\times \dots \times \X_k, \Y\right)$. We use this fact in the following result to completely characterize Birkhoff-James orthogonality in $\B^{k}\left(\X_1\times \dots \times \X_k, \Y\right).$ This is one of the main results of the paper.

\medskip

\begin{theorem}\label{Theorem:1}
Let $\X_i,~\Y$ be normed linear spaces; $1\leq i \leq k,$ and let $\T,~\A \in \B^{k}\left(\X_1\times \dots \times \X_k, \Y\right)$ be non-zero. Then the following are equivalent:

\begin{itemize}
    \item[(i)] $\T\perp_B \A$.
 
 \medskip
 
    \item[(ii)] $0\in \C(\Omega(\T,\A)),$ where $\Omega(\T,\A)$ is the subset of scalars defined by (\ref{Set Omega}).
\end{itemize}
\end{theorem}

\begin{proof}
Without loss of generality, let us assume that $\|\A\|=\|\T\|=1$.

\noindent (i)$\implies$(ii): Suppose if possible $0\notin \C(\Omega(\T,\A)).$ We now proceed in three steps to obtain a contradiction.

\medskip

{\bf Step I:} Since $\C(\Omega(\T,\A))\subseteq \F$ is compact, rotating $\C(\Omega(\T,\A))$ suitably, we may assume that $\R~\lambda > r$  for some $0< r < \dfrac{1}{2}$ and for all $\lambda\in \C(\Omega(\T,\A))$. In particular,
\begin{align}\label{Equation:1}
\R~\lambda > r \quad \forall~\lambda\in \Omega(\T,\A). \end{align}

\medskip

{\bf Step II:} Define 
\[ G:= \left\{(\x, y^*)\in \left(S_{\X_1}\times \dots \times S_{\X_k}\right)\times S_{\Y^*}:~\R~y^*(\T \x)\overline{y^*(\A \x)}\leq r \right\}.\]
\noindent We claim that $\sup \{|y^*(\T \x)|:~(\x,y^*)\in G\}< 1-2\varepsilon$ for some $0< \varepsilon < \dfrac{1}{2}.$ Suppose on contrary  that there exists a sequence $((\x_n,y_n^*))\subseteq G$ so that $|y_n^*(\T \x_n)|\to 1.$ Then we can find a subsequence $(n_j)$ of natural numbers such that
\[ (i)~y^*_{n_j}(\T \x_{n_j})\to \sigma,~|\sigma|=1;~(ii)~{y^*_{n_j}(\A \x_{n_j})}\to \alpha.\]
Put $\widetilde{y}_{n_j}^*=\overline{\sigma}~ y^*_{n_j}$ for each $j\in \mathbb{N}.$ We now have $\widetilde{y}_{n_j}^*(\A \x_{n_j})\to \overline{\sigma}\alpha\in \Omega(\T,\A),$ since $((\x_{n_j},\widetilde{y}^*_{n_j}))\in \M_\T.$ Consequently, $\R~\overline{\sigma}\alpha >r$, by (\ref{Equation:1}). Observe that
\[\R~{y}^*_{n_j}(\T \x_{n_j})\overline{{y}^*_{n_j}(\A \x_{n_j})}=\R~\widetilde{y}^*_{n_j}(\T \x_{n_j})\overline{\widetilde{y}^*_{n_j}(\A \x_{n_j})}\to \R~\sigma\overline{\alpha}>r.\]
However, it follows from the definition of $G$ that
\[\R~{y}^*_{n_j}(\T \x_{n_j})\overline{{y}^*_{n_j}(\A \x_{n_j})}\leq r,\qquad \mathrm{for~all}~j\in \mathbb{N}.\]
Therefore, we arrive at a contradiction and the claim is proved.

\medskip

{\bf Step III:} Choose $0< \mu < \min \left\{\varepsilon, r\right\}. $ Then for any $(\x, y^*)\in G$, we have

\[|y^*(\T \x-\mu \A \x)|< 1-2\varepsilon + \mu< 1-\varepsilon.\]
On the other hand, given any $(\x, y^*)\in \left(\left(S_{\X_1}\times \dots \times S_{\X_k}\right)\times S_{\Y^*}\right)\setminus G$
\[|y^*(\T \x- \mu \A \x)|^2\leq 1+\mu^2-2\mu~\R~ y^*(\T \x)\overline{y^*(\A \x)}\leq 1+\mu^2-2\mu r.\]
\noindent In other words,
\begin{align*}
\|\T-\mu \A\|=\sup \left\{|y^*(\T \x- \mu \A \x)|:~(\x,y^*)\in \left(S_{\X_1}\times \dots \times S_{\X_k}\right)\times S_{\Y^*}\right\}< 1=\|\T\|.
\end{align*}
This is a contradiction as $\T\perp_B \A.$

\medskip

\noindent (ii)$\implies$(i): We first show that for any non-zero scalar $\kappa$, there exists $\lambda\in \Omega(\T,\A)$ such that $\R~\kappa \lambda\geq 0.$ Indeed, since $0\in \C(\Omega(\T,\A))$, by Carath\'{e}odory Theorem there exist $\lambda_i\in \Omega(\T,\A), t_i\in [0,1]$, $1\leq i \leq 3;$ such that
\[ \sum\limits_{i=1}^3 t_i\lambda_i=0,~\sum\limits_{i=1}^3t_i=1.\]
Thus, $\sum\limits_{i=1}^3 t_i\kappa \lambda_i=0$ and at least for one $\lambda_i,$ we must have $\R~\kappa \lambda_i\geq 0.$

\medskip

Let $\gamma\in \F$ be non-zero. By the forgoing discussion, we can find $\lambda_0\in \Omega(\T,\A)$ such that $\R~\gamma\lambda_0 \geq 0$. More precisely, we can find $((\x_n,y_n^*))\in \M_\T$ such that 
\[{y_n^*(\A \x_n)}\to \lambda_0~\mathrm{and}~\R~\gamma~ \lambda_0 \geq 0.\]
So, we have
\begin{align*}
\|\T  + \gamma \A\|^2 & \geq |y_n^*(\T \x_n+\gamma \A \x_n)|^2\\
& = |y_n^*(\T \x_n)|^2+ |\gamma y_n^*(\A \x_n)|^2+2~\R~\gamma~ y_n^*(\A\x_n)\overline{y_n^*(\T \x_n)}\\
& \geq |y_n^*(\T \x_n)|^2+ 2~\R~\gamma~ y_n^*(\A\x_n)\overline{y_n^*(\T \x_n)}.
\end{align*}
Letting $n\to \infty,$ we have
\[|y_n^*(\T \x_n)|^2+ 2~\R~\gamma~ y_n^*(\A\x_n)\overline{y_n^*(\T \x_n)}\to 1+ 2~\R~\gamma~ \lambda_0 \geq 1.\]
Consequently, $\|\T+\gamma\A\|\geq\|\T\|.$ Since $\gamma$ was chosen arbitrarily, we get $\T\perp_B \A$. The proof is now complete.
\end{proof}

\begin{remark}\label{Remark:1}
For $k=1$, the above theorem provides a characterization of orthogonality in $\B^{1}(\X,\Y)$ (in both real and complex setting), and to the best of our knowledge, it is a new addition to the existing literature \cite{BG, BS, Roy & Bagchi, RSS, Sain, Sain3, SPM, Turnsek}.
\end{remark}

We have an easier description of $\Omega(\T, \A)$ in the finite-dimensional context. Hence, Theorem \ref{Theorem:1} takes a simpler form. Let us discuss this fact in more detail in the following remark.

\medskip

\begin{remark}\label{Remark:2}
Let $\T,~\A\in \B^{k}\left(\X_1\times \dots \times \X_k, \Y\right)$ be non-zero, where $\X_i,~\Y$ are finite-dimensional Banach spaces $(1\leq i \leq k)$. Consider the collection
\begin{align}\label{Set Omega prime}
\Omega'(\T,\A):=\left\{{y^*(\A \x)}:~\x\in M_\T,~y^*\in J(\T \x)\right\}\subseteq \F.
\end{align}
Given any $\lambda\in \Omega(\T,\A),$ there exists $((\x_n, y_n^*))\in \M_\T$ such that
\[ y_n^*(\T \x_n)\to \|\T\|~\mathrm{and}~{y_n^*(\A \x_n)}\to \lambda.\]
\noindent Let $\pi_i$ denote the $i$-th projection map from $\X_1\times \dots \times \X_k$ into the factor space $\X_i$ for each $i=1,\dots,k.$ Then it is easy to see that $((\pi_i(\x_n))$ has a convergent subsequence for $1\leq i \leq k.$ So we can find a subsequence $(n_p)$ of natural numbers such that $\pi_i\left(\x_{n_p}\right)\to x_i$ for some $x_i\in S_{\X_i}$, $y_{n_p}^*\to y_0^*$ for some $y_0^*\in S_{\Y^*}$ and $1\leq i \leq k.$ Let $\x_0=(x_1, \dots, x_k)$. Evidently $\x_{n_p}\to \x_0.$ Also, due to the continuity of $\T$, we have $\T\x_{n_p}\to \T\x_0$. A straightforward computation gives that
\begin{align*}
y_{n_p}^*(\T \x_{n_p})\to y_0^*(\T \x_0)=\|\T\|\quad \mathrm{and} \quad~{y_{n_p}^*(\A \x_{n_p})}\to {y_0^*(\A \x_0)}=\lambda.
\end{align*}
The above expressions show the following:
\begin{align*}
(i)~\x_0\in M_\T,~(ii)~ y_0^*\in J(\T \x_0),~(iii)~{y_0^*(\A \x_0)}=\lambda.
\end{align*}
Therefore, $\lambda \in \Omega'(\T, \A)$. On the other hand, for any $\lambda'\in \Omega'(\T, \A)$, there exist $\x \in M_\T$ and $y^*\in J(\T \x)$ such that $\lambda'={y^*(\A \x)}.$ Choosing $\x_n = \x$ and $y_n^*=y^*$ for all $n\in \mathbb{N}$, we have $\lambda'\in \Omega(\T,\A).$ This shows that $\Omega(\T,\A)=\Omega'(\T,\A).$ 
\end{remark}

\medskip

The following corollary is an immediate consequence of Theorem \ref{Theorem:1} and Remark \ref{Remark:2} and thus we omit the proof.

\begin{cor}\label{Corollary:1}
Let $\X_i,~\Y$ be finite-dimensional Banach spaces; $1\leq i \leq k,$ and let $\T,~\A \in \B^{k}\left(\X_1\times \dots \times \X_k, \Y\right)$ be non-zero. Then the following are equivalent:

\begin{itemize}
    \item[(i)] $\T\perp_B \A$.

\medskip
    
    \item[(ii)] $0\in \C\left(\Omega'(\T,\A)\right),$ where $\Omega'(\T,\A)$ is the subset of scalars defined by (\ref{Set Omega prime}).
\end{itemize}
\end{cor}

\medskip

Let $\HS$ be an $n$-dimensional Hilbert space. We can think of $\B^{1}(\HS)$ as the space of all $n\times n$ matrices, identified as operators acting on $\HS$ in the usual way. For $k=1$ and $\X=\Y=\HS,$ Corollary \ref{Corollary:1} gives a necessary and sufficient condition of Birkhoff-James orthogonality in $\B^{1}(\HS)$. Moreover, in this case, Corollary \ref{Corollary:1} reduces to the well-known Bhatia-\v{S}emrl Theorem, once we apply the Toeplitz-Hausdorff Theorem \cite{Gustafson}. We sketch the outline of the proof in the following remark.

\medskip

\begin{remark}\label{Remark:3}
Let $\T,~\A$ be two non-zero $n\times n$ matrices acting on $\HS$. It follows from Riesz representation Theorem that given any non-zero $x\in \HS$, the unique support functional $y^*$ at $x$ is of the form $y^*(z)= \left\langle z, \dfrac{x}{\|x\|}\right\rangle$ for all $z\in \HS.$ Now, it is not difficult to see that 
\[\Omega'(\A,\T)= \left\{\left\langle \A x, \frac{\T x}{\|\T\|} \right\rangle:~x\in M_\T\right\}.\]
On the other hand, it is well-known \cite{Sain Paul} that $M_\T= S_{\HS_0}$, where $\HS_0$ is the eigen space of $\T^*\T$ corresponding to the highest eigen value $\|\T\|^2$. Therefore, we have
\[\left\{\left\langle \A x, \frac{\T x}{\|\T\|} \right\rangle:~x\in M_\T\right\}=\left\{\left\langle \A x, \frac{\T x}{\|\T\|} \right\rangle:~x\in S_{\HS_0}\right\},\]
which is convex by the Toeplitz-Hausdorff Theorem \cite{Gustafson}. Consequently, it follows from Corollary \ref{Corollary:1} that 
\[ \T\perp_B \A ~\iff~\exists ~x_0\in S_\HS~\mathrm{such~that}~\|\T x_0\|=\|\T\|~\mathrm{and}~\langle \A x_0, \T x_0 \rangle=0.\]
\end{remark}

\medskip

One can expect more sophisticated characterization of orthogonality in $\B^{k}\left(\HS_1\times \dots \times \HS_k, \HS\right)$ due to the inner product structure of a Hilbert space. For this reason, given any $\T,\A\in \B^{k}\left(\HS_1\times \dots \times \HS_k, \HS\right),$ below we consider another (non-empty) collection of scalars:
\begin{align}\label{Set W(T,A)}
\mathrm{W}(\T,\A):=\left\{\beta\in \F:~\langle \A\x_n, \T\x_n \rangle \to \beta,~(\x_n)\subseteq S_{\HS_1}\times \dots \times S_{\HS_k},~\|\T\x_n\|\to \|\T\|\right\}.
\end{align}

\medskip

The following result completely characterizes Birkhoff-James orthogonality of multilinear maps in the setting of infinite-dimensional Hilbert spaces.

\begin{theorem}\label{Theorem:2}
Let $\HS_i,~\HS$ be Hilbert spaces, $1\leq i \leq k,$ and let $\T,~\A \in \B^{k}\left(\HS_1\times \dots \times \HS_k, \HS\right)$ be non-zero. Then the following are equivalent:

\begin{itemize}
    \item[(i)] $\T\perp_B \A$.
 
 \medskip
    
    \item[(ii)] $0\in \C(\Omega(\T,\A))$, where $\Omega(\T,\A)$ is the subset of scalars defined by (\ref{Set Omega}).

\medskip

\item[(iii)] $0\in \C(\mathrm{W}(\T,\A))$, where $\mathrm{W}(\T,\A)$ is the subset of scalars defined by (\ref{Set W(T,A)}).
\end{itemize}
\end{theorem}

\begin{proof}
Without loss of generality, assume that $\|\T\|=\|\A\|=1.$ Note that the equivalence of (i) and (ii) is already established in Theorem \ref{Theorem:1}. So we only prove (ii)$\implies$(iii) and (iii)$\implies$(i).

\medskip

\noindent (ii)$\implies$(iii): It is enough to show that $\Omega(\T,\A)\subseteq \mathrm{W}(\T,\A).$ Let $\lambda\in \Omega(\T,\A).$ Then there exists a sequence $((\x_n,y_n^*))\in \M_\T$ so that 
\[ (i)~y_n^*(\T \x_n)\to 1,~(ii)~{y_n^*(\A\x_n)}\to \lambda.\]
Evidently, $\|\T\x_n\|\to 1.$ By Riesz representation Theorem, for each $n\in \mathbb{N}$, there exists a unique $y_{n}\in S_\HS$ such that 
\[y_{n}^*(z)=\langle z, y_{n} \rangle, \qquad \forall ~z\in \HS.\]

Now,
\begin{align*}
\left\| y_{n}-{\T\x_{n}}\right\|^2 & = \langle y_n,y_n\rangle-\langle y_n, \T\x_n\rangle-\langle \T\x_n, y_n \rangle+\langle \T\x_n,\T\x_n \rangle\\
& = 1-2~\R~y_{n}^*(\T\x_{n})+\|\T\x_n\|^2\\
& \longrightarrow 0,~\mathrm{as}~n\to \infty.
\end{align*}

Thus, 
\begin{align*}
|\langle \A\x_n,\T\x_n\rangle-y_n^*(\A\x_n)| & =|\langle \A\x_n,\T\x_n\rangle-\langle \A\x_n,y_n \rangle|\\
& =|\langle \A\x_n, \T\x_n-y_n\rangle|\\
& \leq \|\T\x_n-y_n\|\to 0,~\mathrm{as}~n\to \infty.
\end{align*}
Therefore, $\langle \A\x_n,\T\x_n\rangle\to \lambda\in \mathrm{W}(\T,\A)$, and $\Omega(\T,\A)\subseteq \mathrm{W}(\T,\A).$

\medskip

\noindent (iii)$\implies$(i): Let $\mu \in \mathbb{F}$ be non-zero. By the same arguments as in the proof ((ii)$\implies$(i)) of Theorem \ref{Theorem:1}, we can find $(\x_n) \subseteq S_{\HS_1}\times \dots \times S_{\HS_k}$ such that
\[(i)~\|\T \x_n\|\to \|\T\|,~(ii)~\langle \A\x_n, \T\x_n \rangle \to \beta_0,~(iii)~\R~\mu~\beta_0 \geq 0.\] 
Therefore, we get
\begin{align*}
\|\T + \mu \A\|^2 & \geq  \langle \T \x_n+ \mu \A\x_n, \T\x_n + \mu \A \x_n \rangle\\
& =  \|\T \x_n\|^2+ \|\mu \A \x_n\|^2+ 2~\R~\mu~\langle \A\x_n, \T\x_n \rangle\\
& \geq  \|\T \x_n\|^2 +   2~\R~\mu~\langle \A\x_n, \T\x_n \rangle.
\end{align*}
Letting $n\to \infty,$ we have 
\[\|\T \x_n\|^2 +   2~\R~\mu~\langle \A\x_n, \T\x_n \rangle\to 1+2~\R~\mu~\beta_0 \geq 1.\]
\noindent Consequently, $\|\T+\mu\A\|\geq\|\T\|.$ Since $\mu$ was chosen arbitrarily, we have $\T\perp_B \A$ and the proof is complete.
\end{proof}

\begin{remark}\label{Remark:4}
In Theorem \ref{Theorem:2}, we assume $k=1$ and $\HS_1=\HS.$ Let $\T,~\A\in \B^{1}(\HS)$ be non-zero. Let us consider the following collection of scalars, $\mathrm{W}_0(\T,\A)$, defined by
\begin{align*}
\mathrm{W}_0(\T,\A):=\{\beta\in \F:~\langle \T x_n, \A x_n \rangle \to \beta,~\|x_n\|=1,~\|\T x_n\|\to \|\T\|\}.
\end{align*}

\noindent The above collection, in the special case when $\A=\mathbf{I}$, is popularly known as the \emph{Maximal numerical range} of $\T$ and is a convex subset of $\F$ (see \cite{Stampfli}). Note that $\mathrm{W}_0(\T,\A)$ is non-empty, closed and convex [\cite{PHD},Theorem 2]. It is now easy to see (for $k=1$ and $\HS_1=\HS$) that
\[ \mathrm{W}(\T,\A)=\left\{\overline{\lambda}:~\lambda\in \mathrm{W}_0(\T,\A)\right\}.\]
This establishes the convexity of $\mathrm{W}(\T,\A)$ (see (\ref{Set W(T,A)})). Thus, for $\T,\A\in \B^1(\HS)$, $\C(\mathrm{W}(\T,\A))$ is $\mathrm{W}(\T,\A)$ itself.
\end{remark}

\medskip

Theorem \ref{Theorem:2} in combination with the above remark facilitates the characterization of operator orthogonality in the Hilbert space setting, which is famously known as Bhatia-\v{S}emrl Theorem \cite{BS}. In that sense, Theorem \ref{Theorem:2} can be regarded as a multilinear generalization of Bhatia-\v{S}emrl Theorem.

\begin{theorem}(Bhatia-\v{S}emrl Theorem in infinite-dimension)
Let $\HS$ be a Hilbert space and let $\T,~\A\in \B^{1}(\HS)$ be non-zero. Then the following are equivalent:

\begin{itemize}
    \item[(i)] $\T\perp_B \A.$

 \medskip
    
    \item[(ii)] There exists $(x_n)\subseteq S_\HS$ such that $\|\T x_n\|\to \|\T\|$ and $\langle \A x_n, \T x_n \rangle \to 0.$
\end{itemize}
\end{theorem}

\vspace{1mm}

\section{Weak differentiability of the norm}

\vspace{2mm}

We devote this section to characterize the weak differentiability of the norm of $\B^{k}(\X_1\times \dots \times \X_k, \Y).$  Most of the results of this section are dependent on the results of the previous section. Let $\mathbb{S}$ denote the unit sphere of the scalar field $\F$. The symbol $\mathbb{S}^k$ denotes the product $\underbrace{\mathbb{S}\times \dots \times \mathbb{S}}_{k-times}$. The following theorem completely characterizes the weak differentiable points in $\B^{k}(\X_1\times \dots \times \X_k, \Y)$ and is the main result of this paper.

\begin{theorem}\label{Theorem:3}
Let $\X_i,~\Y$ be normed linear spaces, for $1\leq i \leq k,$ and let $\T\in \B^{k}\left(\X_1\times \dots \times \X_k, \Y\right)$ be non-zero. Then the following are equivalent:

\begin{itemize}
    \item[(i)] The norm of $\B^{k}\left(\X_1\times \dots \times \X_k, \Y\right)$ is weakly differentiable at $\T$.

    \item[(ii)] $\Omega(\T, \A)$ is a singleton set for any $\A\in \B^{k}\left(\X_1\times \dots \times \X_k, \Y\right)$, where $\Omega(\T,\A)$ is the subset of scalars defined in (\ref{Set Omega}).
\end{itemize}
\end{theorem}

\begin{proof}
Without loss of generality, assume that $\|\T\|=\|\A\|=1.$

\noindent (i)$\implies$(ii): Suppose on contrary that $\Omega(\T,\A)$ is not a singleton set for some $\A\in \B^{k}\left(\X_1\times \dots \times \X_k, \Y\right)$.  Let $\alpha$, $\beta$ be two distinct members of $\Omega(\T,\A).$ Then there exist $((\x_n,y_n^*))\in \M_\T$ and $((\z_n,w_n^*))\in \M_\T$ such that the following hold:
\[ (i)~{y_n^*(\A\x_n)}\to \alpha,~(ii)~{w_n^*(\A\z_n)}\to \beta.\]
Now, there are two possibilities.

\medskip

\noindent Case I: Suppose that $\alpha, \beta \neq 0.$ Let $\mathsf{U}_1, \mathsf{U}_2\in \B^{k}\left(\X_1\times \dots \times \X_k, \Y\right)$ be defined by

\[ \mathsf{U}_1:=\T-\frac{1}{{\alpha}}\A,~\mathsf{U}_2:=\T-\frac{1}{{\beta}}\A.\]
Then we get
\begin{align*}
\lim y_n^*(\mathsf{U}_1\x_n) =\lim y_n^*(\T\x_n)-\frac{1}{\alpha}\lim y_n^*\left(\A \x_n\right)= 0.
\end{align*}

\medskip

\noindent Similarly, $w_n^*(\mathsf{U}_2\z_n)\to 0.$ Thus, $0\in \Omega(\T, \mathsf{U}_1)\cap \Omega(\T, \mathsf{U}_2).$ It now follows from Theorem \ref{Theorem:1} that $\T\perp_B \mathsf{U}_1$ and $\T\perp_B \mathsf{U}_2.$ By the homogeneity of Birkhoff-James orthogonality, we get 
\[\T\perp_B \dfrac{-{\alpha}}{({\beta}-{\alpha})}\mathsf{U}_1 \quad \mathrm{and} \quad \T\perp_B \dfrac{{\beta}}{({\beta}-{\alpha})}\mathsf{U}_2.\]
Since the norm of $\B^{k}\left(\X_1\times \dots \times \X_k, \Y\right)$ is weakly differentiable at $T$, it follows from Theorem \ref{Primary Result} that Birkhoff-James orthogonality is right-additive at $\T$. Thus, we have

\[ \T\perp_B \left(\dfrac{-{\alpha}}{({\beta}-{\alpha})}\mathsf{U}_1+\dfrac{{\beta}}{({\beta}-{\alpha})}\mathsf{U}_2\right).\]

\medskip

This proves that $\T\perp_B \T$, which is a contradiction.

\medskip

\noindent Case II: At least one of $\alpha, \beta$ is zero. Without loss of generality, let $\alpha=0.$ Then $\T\perp_B \A$. Also, by Case I, $\T\perp_B \left( \T-\dfrac{1}{{\beta}}\A\right).$ Therefore, we have

\[ \T\perp_B \left(\frac{1}{{\beta}}\A+\T-\frac{1}{{\beta}}\A\right).\]
Thus, $\T\perp_B \T$ which is again a contradiction.

\medskip

\noindent (ii)$\implies$(i): Consider $\A_1, \A_2 \in \B^{k}\left(\X_1\times \dots \times \X_k, \Y\right)$ with $\T\perp_B \A_1$, $\T\perp_B \A_2$. Therefore, $0\in \C(\Omega(\T,\A_1))\cap \C(\Omega(\T,\A_2))$, by Theorem \ref{Theorem:1}. On the other hand, $\Omega(\T, \A_1)$ and $\Omega(\T,\A_2)$ are singleton sets. This implies $\Omega(\T,\A_1)=\Omega(\T,\A_2)=\{0\}.$
Thus, given any $((\x_n,y_n^*))\in \M_\T$, we can find a common subsequence $(n_j)$ of natural numbers so that $y_{n_j}^*(\A_1 \x_{n_j}) \to 0,~ y_{n_j}^*(\A_2 \x_{n_j}) \to 0,$ as $j\to \infty$. Since $((\x_{n_j},y_{n_j}^*))$ is also a member of $\M_\T$, and $y_{n_j}^*((\A_1+\A_2)\x_{n_j})) \to 0$, we have $0\in \Omega(\T,\A_1+\A_2)\subseteq \C(\Omega(\T,\A_1+\A_2))$. Therefore, by Theorem \ref{Theorem:1}, $\T\perp_B (\A_1+\A_2)$ and consequently, the proof follows from Theorem \ref{Primary Result}.
\end{proof}

As a corollary of Theorem \ref{Theorem:3}, we have the following conclusive result regarding the weak differentiability of the norm of $\B^{k}\left(\X_1\times \dots \times \X_k, \Y\right)$.

\begin{theorem}
Let $\X_i,~\Y$ be normed linear spaces; $1\leq i \leq k,$ and let $\T\in \B^{k}\left(\X_1\times \dots \times \X_k, \Y\right)$ be non-zero. Then the following are equivalent:

\begin{itemize}
    \item[(i)] The norm of $\B^{k}\left(\X_1\times \dots \times \X_k, \Y\right)$ is weakly differentiable at $\T$.

    \medskip
    
    \item[(ii)] For any two members $((\x_n,y_n^*))$ and $((\z_n,w_n^*))$ of $\M_\T$ and for any $\A\in \B^{k}\left(\X_1\times \dots \times \X_k, \Y\right)$, both $\lim y_n^*(\A \x_n)$ and $\lim w_n^*(\A \z_n)$ exist, and are equal.
\end{itemize}
\end{theorem}

\begin{proof}
(i)$\implies$(ii): We first note that

\medskip

\begin{itemize}
    \item $\lim y_n^*(\A \x_n)$ and $\lim w_n^*(\A \z_n)$, whenever exist, are members of $\Omega(\T,\A)$.

\medskip    
    
    \item Every subsequential limit of the sequences $(y_n^*(\A \x_n))$ and $(w_n^*(\A \z_n))$ are also members of $\Omega(\T,\A)$.

\medskip    
    
    \item Since the norm is weakly differentiable at $\T$, $\Omega(\T,\A)=\{\lambda\}$ for some $\lambda\in \F$, by Theorem \ref{Theorem:3}.
\end{itemize}

\medskip

\noindent Clearly, $(y_n^*(\A \x_n))$ is a bounded sequence of scalars, and therefore, has a convergent subsequence. For any convergent subsequence $\left(y_{n_p}^*(\A \x_{n_p})\right)$ of $(y_n^*(\A \x_n))$, we have $\lim y_{n_p}^*(\A \x_{n_p})= \lambda$. Since any convergent subsequence of the bounded sequence $(y_n^*(\A \x_n))$ converges to $\lambda,$ the sequence $(y_n^*(\A \x_n))$ itself converges to $\lambda.$ Similarly, $\lim w_n^*(\A \z_n)=\lambda$.

\medskip

\noindent (ii)$\implies$(i): The stated condition indicates that $\Omega(\T,\A)$ is a singleton set for any $\A\in \B^{k}\left(\X_1\times \dots \times \X_k, \Y\right)$. The rest of the proof follows from Theorem \ref{Theorem:3}.
\end{proof}

\medskip

Theorem \ref{Theorem:3} can also be stated in terms of semi-inner-products. We illustrate this in our next result. Let us recall the definition of semi-inner-product \cite{Giles, Lumer, Sain3} in this connection.

\begin{definition}
Let $\X$ be a vector space over $\F.$ A semi-inner-product (SIP) on $\X$ is a function $[\cdot, \cdot]:\X \times \X\to \F$ that satisfies the following conditions:
\begin{itemize}
    \item $[x,x]>0$ for all non-zero $x\in \X.$
    
    \medskip

    \item $[\alpha y+\beta z,x]=\alpha [y,x]+\beta [z,x]$ for all $x,y,z\in \X,$ and $\alpha, \beta \in \F.$
    
    \medskip

    \item $[y,\lambda x]=\overline{\lambda}[y,x]$ for all $x,y\in \X$ and $\lambda\in \F.$
    
    \medskip
    
    \item $|[y,x]|^2\leq \|y\|^2\|x\|^2$ for all $x,y\in \X.$
\end{itemize}
\end{definition}

\medskip

The vector space $\X$, equipped with an SIP, is called an SIP space. Every SIP space is a normed linear space with the norm $\|x\|=[x,x]^{\frac{1}{2}}.$ On the other hand, every normed linear space can be made into an SIP space, in general, in infinitely many different ways.

\begin{Lemma}\label{Lemma:2}
Let $\X_i,~\Y$ be normed linear spaces for $1\leq i \leq k,$ and let $\T,~\A\in \B^{k}\left(\X_1\times \dots \times \X_k, \Y\right)$ be non-zero. Let $\mathcal{L}$ denote the collection of all SIPs on $\Y.$ Let $\Lambda(\T,\A)$ be the (non-empty) collection of scalars defined by
\begin{align}\label{Set:Lambda}
\Lambda(\T,\A):=\left\{\lambda\in \F:[\A\x_n,\T\x_n]_n\to \lambda,([\cdot,\cdot]_n)\subseteq \mathcal{L},~(\x_n)\subseteq S_{\X_1}\times \dots \times S_{\X_k},~\|\T\x_n\|\to \|\T\|\right\}.
\end{align}
Then for any $\lambda\in \Lambda(\T,\A)$, $\dfrac{\lambda}{\|\T\|}$ is a member of $\Omega(\T,\A),$ where $\Omega(\T,\A)$ is the subset of scalars defined by (\ref{Set Omega}).
\end{Lemma}

\begin{proof}
Let $\lambda\in \Lambda(\T,\A)$. Then, there exist a sequence $(\x_n)\subseteq S_{\X_1}\times \dots \times S_{\X_k}$ with $\|\T\x_n\|\to \|\T\|$ and a sequence of SIPs $([\cdot, \cdot]_n)\subseteq \mathcal{L}$ such that $[\A\x_n,\T\x_n]_n\to \lambda.$ Choose $N\in \mathbb{N}$ be such that $\|\T\x_n\|>0$ for all $n> N.$ Let $\z_m=\x_{_{N+m}}$ for all $m\in \mathbb{N}.$ For each $m\in \mathbb{N}$, define $y_m^*:\Y\to \F$ by
\[ y_m^*(y)=\frac{1}{\left\|\T\z_m\right\|}\left[y,\T\z_m\right]_{N+m},\qquad \mathrm{for~all}~y\in \Y.\]
Then $y_m^*\in S_{\Y^*}$ and $y_m^*(\T\z_m)\to \|\T\|$. Consequently, $(\z_m,y_m^*)\in \M_\T.$ It is now obvious that $\lim y_m^*(\A\z_m)=\dfrac{\lambda}{\|\T\|}.$ Thus, $\dfrac{\lambda}{\|\T\|}\in \Omega(\T,\A)$ and the proof is complete.
\end{proof}

Using Lemma \ref{Lemma:2}, we now reformulate Theorem \ref{Theorem:3} in terms of SIP.

\begin{theorem}\label{Theorem:4}
Let $\X_i,~\Y$ be normed linear spaces for $1\leq i \leq k,$ and let $\T\in \B^{k}\left(\X_1\times \dots \times \X_k, \Y\right)$ be non-zero. Let $\mathcal{L}$ denote the collection of all SIPs on $\Y.$ Then the following are equivalent.

\begin{itemize}
    \item[(i)] The norm of $\B^{k}\left(\X_1\times \dots \times \X_k, \Y\right)$ is weakly differentiable at $\T$.

    \medskip
    
    \item[(ii)] $\Omega(\T, \A)$ is a singleton set for any $\A\in \B^{k}\left(\X_1\times \dots \times \X_k, \Y\right)$, where $\Omega(\T,\A)$ is the subset of scalars defined by (\ref{Set Omega}).
    
    \medskip
    
    \item[(iii)] If $\T\perp_B \A$ for some $\A\in \B^{k}\left(\X_1\times \dots \times \X_k, \Y\right),$ then $[\A\x_n,\T\x_n]_n\to 0$ for any sequence $(\x_n)\subseteq S_{\X_1}\times \dots \times S_{\X_k}$ with $\|\T\x_n\|\to \|\T\|$ and for any sequence of SIPs $([\cdot,\cdot]_n)\subseteq \mathcal{L}$.
\end{itemize}
\end{theorem}

\begin{proof}
We only prove (ii)$\implies$(iii) and (iii)$\implies$(i), as (i)$\iff$(ii) follows from Theorem \ref{Theorem:3}.

\medskip

\noindent (ii)$\implies$(iii): Since $\T\perp_B \A$ and $\Omega(\T,\A)$ is a singleton set, we already know from Theorem \ref{Theorem:3} that $\Omega(\T,\A)=\{0\}.$ Let $(\x_n)\subseteq S_{\X_1}\times \dots \times S_{\X_k}$ be any sequence with $\|\T\x_n\|\to \|\T\|$ and let $([\cdot,\cdot]_n)\subseteq \mathcal{L}$ be any sequence of SIPs. Consider any convergent subsequence $\left(\left[\A\x_{n_p},\T\x_{n_p}\right]_{n_p}\right)$ of the bounded sequence $([\A\x_n,\T\x_n]_n)$, and let $\left[\A\x_{n_p},\T\x_{n_p}\right]_{n_p}\to \lambda.$ Now, by Lemma \ref{Lemma:2}, we have $\dfrac{\lambda}{\|\T\|}\in \Omega(\T,\A)$. Thus, we have $\lambda=0.$ Since any convergent subsequence of $([\A\x_n,\T\x_n]_n)$ converges to $0$, the sequence $([\A\x_n,\T\x_n]_n)$ converges to $0.$

\medskip

\noindent (iii)$\implies$(i): Let $\A_1,\A_2\in \B^{k}\left(\X_1\times \dots \times \X_k, \Y\right)$ with $\T\perp_B \A_1$, $\T\perp_B \A_2.$ Let $(\x_n)\subseteq S_{\X_1}\times \dots \times S_{\X_k}$ with $\|\T\x_n\|\to \|\T\|$ and let $([\cdot,\cdot]_n)\subseteq \mathcal{L}$ be any sequence of SIPs. Note that $[\A_1\x_n+\A_2\x_n,\T\x_n]_n=[\A_1\x_n,\T\x_n]_n+[\A_2\x_n,\T\x_n]_n$, for all $n\in \mathbb{N}.$ It now follows from the hypothesis that $[\A_1\x_n+\A_2\x_n,\T\x_n]_n\to 0.$ Moreover, for any scalar $\lambda$, we have
\begin{align*}
\|\T+\lambda(\A_1+\A_2)\|\|\T\|&\geq \|\T\x_n+\lambda (\A_1\x_n+\A_2\x_n)\|\|\T\x_n\|\\
&\geq \left|[\T\x_n+\lambda\A_1\x_n+\lambda \A_2\x_n,\T\x_n]_n\right|\\
& = \left|[\T\x_n,\T\x_n]_n+\lambda[\A_1\x_n+\A_2\x_n,\T\x_n]_n\right|\\
& \to \|\T\|^2,~\mathrm{as}~n\to \infty.
\end{align*}
In other words, $\T\perp_B (\A_1+\A_2)$. Consequently, the proof follows from Theorem \ref{Primary Result}.
\end{proof}

\begin{remark}\label{Remark:5}
For $k=1$, the above theorem reduces to the characterization of the weak differentiability of norm on the spaces of bounded linear operators. Note that Theorem \ref{Theorem:4} improves Theorem 2.1 in \cite{SPMR}.
\end{remark}

\medskip

We now turn our attention towards the finite-dimensional case. An additional advantage in this case is that $M_\T\neq \emptyset$ here. Before that let us state a result which is an obvious corollary of Theorem \ref{Theorem:3} and a proof of this follows from Remark \ref{Remark:2}.

\medskip

\begin{cor}\label{Corollary:4}
Let $\X_i,~\Y$ be finite-dimensional Banach spaces for $1\leq i \leq k,$ and let $\T\in \B^{k}\left(\X_1\times \dots \times \X_k, \Y\right)$ be non-zero. Then the following are equivalent:

\begin{itemize}
   \item[(i)] The norm of $\B^{k}\left(\X_1\times \dots \times \X_k, \Y\right)$ is weakly differentiable at $\T$.

\medskip
    
    \item[(ii)] $\Omega'(\T,\A)$ is a singleton set for any $\A\in \B^{k}\left(\X_1\times \dots \times \X_k, \Y\right)$, where $\Omega'(\T,\A)$ is the subset of scalars defined by (\ref{Set Omega prime}).
\end{itemize}
\end{cor}

\medskip

It turns out that the norm attainment set of a $k$-linear map $\T$ plays a decisive role in determining the smoothness of $\T.$ This is illustrated in the next theorem.

\begin{theorem}\label{Theorem:5}
Let $\X_i,~\Y$ be finite-dimensional Banach spaces; $1\leq i \leq k,$ and let $\T\in \B^{k}\left(\X_1\times \dots \times \X_k, \Y\right)$ be non-zero. Then the following are equivalent:

\begin{itemize}
   \item[(i)] The norm of $\B^{k}\left(\X_1\times \dots \times \X_k, \Y\right)$ is weakly differentiable at $\T$.  
   
   \medskip

    \item[(ii)] $M_\T:=\left\{ \left(\mu^{(1)} x_1,\dots, \mu^{(k)} x_k\right)\in S_{\X_1}\times \dots \times S_{\X_k}:~\left(\mu^{(1)},    \dots, \mu^{(k)}\right)\in \mathbb{S}^k\right\}$ and the norm of $\Y$ is weakly differentiable at $\T\x_0$, where $\x_0=(x_1, \dots,x_k).$
    
    \medskip
    
    \item[(iii)] $\Omega'(\T,\A)$ is a singleton set for any $\A\in \B^{k}\left(\X_1\times \dots \times \X_k, \Y\right)$, where $\Omega'(\T,\A)$ is the subset of scalars defined by (\ref{Set Omega prime}). To be more precise, $\Omega'(\T,\A)=\{y_0^*(\A \x_0)\}$, where $y_0^*\in J(\T \x_0)$.
\end{itemize}
\end{theorem}

\begin{proof}
We only prove (i)$\implies$(ii) and (ii)$\implies$(iii), as (iii)$\iff$(i) follows from Corollary \ref{Corollary:4}.

\medskip

\noindent (i)$\implies$(ii): Suppose on the contrary that $M_\T$ is not of the stated form. Then there exists $\z_0\in M_\T$ such that $\z_0=(z_1, \dots, z_k)\neq \left(\mu^{(1)} x_1,\dots, \mu^{(k)} x_k\right)$ for any $\left(\mu^{(1)}, \dots, \mu^{(k)}\right)\in \mathbb{S}^k.$ Therefore, there exists some $j\in \{1,\dots,k\}$  such that $z_j\neq \mu x_j$ for any unimodular scalar $\mu.$ In particular, $\{x_j,z_j\}$ is a linearly independent subset of $\X_j$. For all such indices $j\in \{1,2, \dots, k\}$ such that $\{x_j, z_j\}$ is linearly independent, we extend $\{x_j,z_j\}$ to a basis of $\X_j$. Also, for all such indices $m\in \{1,2, \dots, k\}$ such that $\{x_m, z_m\}$ is linearly dependent, we extend $\{x_m\}$ to a basis of $\X_m$. In this way, we get a basis, say $B_i$, of $\X_i$ for each $1\leq i \leq k$. Let
\[\mathbf{v}_0:=(v_1, \dots, v_k)~\mathrm{where}~
v_j=\begin{cases}
z_j,~&\mathrm{if}~\{x_j,z_j\}~ \mathrm{~is~linearly~independent},\\
x_j,~& \mathrm{otherwise}.
\end{cases}
\]

\noindent It can be shown with a little computation that $\x_0,\mathbf{v}_0\in M_\T\cap \left(B_1\times \dots \times B_k\right).$ Note that any (bounded) $k$-linear map $\A:\X_1\times \dots \times \X_k \to \Y$ is determined by its action on $B_1\times \dots \times B_k.$ Now, we construct two  $k$-linear maps  $\A_1$ and $\A_2$ in the following way.

\noindent For any $\x\in B_1\times \dots \times B_k$,

\medskip

$\A_1 \x=\begin{cases}
\theta, &~\mathrm{if}~\x=\x_0,\\
\T \mathbf{v}_0,&~\mathrm{if}~\x=\mathbf{v}_0,\\
\dfrac{1}{2}\T \x,&~\mathrm{otherwise}.
\end{cases}$\qquad 
$\A_2 \x=\begin{cases}
\T \x_0, &~\mathrm{if}~\x=\x_0,\\
\theta,&~\mathrm{if}~\x=\mathbf{v}_0,\\
\dfrac{1}{2}\T \x,&~\mathrm{otherwise}.
\end{cases}$

\medskip

\noindent Evidently $\T\perp_B \A_1$ and $\T \perp_B \A_2.$ Also, it follows from Theorem \ref{Primary Result} that Birkhoff-James orthogonality is right-additive at $\T$. Therefore, we have $\T\perp_B (\A_1+\A_2)$ and this is a contradiction as $(\A_1+\A_2)=\T.$

\medskip

Next, suppose that the norm of $\Y$ is not weakly differentiable at $\T\x_0$, i.e., $\T \x_0$ is not a smooth point. Then, we can find two distinct support functionals $y_1^*,~y_2^*$ at $\T\x_0.$ Let $\phi, \psi:\B^{k}\left(\X_1\times \dots \times \X_k, \Y\right) \to \mathbb{F}$ be two linear functionals defined by
\[ \phi(\mathsf{U})=y_1^*(\mathsf{U} \x_0),~\psi(\mathsf{U})=y_2^*(\mathsf{U}\x_0), \qquad \mathrm{for~all}~\mathsf{U}\in \B^{k}\left(\X_1\times \dots \times \X_k, \Y\right).\]
Since $\ker y_1^* \neq \ker y_2^*$, it is straightforward that $\phi$ and $\psi$ are distinct support functionals at $\T$. Thus, $\T$ is not smooth and this is a contradiction to Theorem \ref{Primary Result}.

\medskip

\noindent (ii)$\implies$(iii): Let $\A\in \B^{k}\left(\X_1\times \dots \times \X_k, \Y\right)$ be arbitrary. Consider any $\x\in M_\T$. Then $\x=(\mu_1 x_1,\dots,\mu_k x_k)$ for some $(\mu_1,\dots,\mu_k)\in \mathbb{S}^k$. Since $\mathbb{S}$ is a group under multiplication and $\T,\A$ are $k$-linear, we get
\[\T\x=\mu\T\x_0,~\A\x=\mu\A\x_0, \qquad \mathrm{for~some~unimodular~scalar}~\mu.\]
Also, since norm of $\Y$ is weakly differentiable at $\T\x_0$, $\T \x_0$ is smooth. Consequently, $\T \x$ is smooth. Therefore, the unique support functional $y^*$ at $\T\x$ is of the form $\overline{\mu} y_0^*$, where $y_0^*$ is the unique support functional at $\T \x_0$. Consequently, we have
\[{y^*(\A \x)} = {\overline{\mu}y_0^*(\mu\A \x_0)}=y_0^*(\A\x_0).\]
\noindent To be more precise, we have
\[\Omega'(\T,\A)=\{y_0^*(\A \x_0)\},\]
where $\x_0\in M_\T$ and $y_0^*\in J(\T\x_0)$. In other words, $\Omega'(\T,\A)$ is a singleton set.
\end{proof}

\begin{remark}\label{Remark:6}
For $k=2$, the above theorem reduces to the characterization of weak differentiability of the norm on the spaces of bounded bilinear operators. In particular, it simplifies and extends Theorem 2.5 in \cite{Sain2}.
\end{remark}

\medskip

A bounded linear operator $\T\in \B^{1}(\X,\Y)$ is said to have the \emph{Bhatia-\v{S}emrl property} \cite{Choi, SPH}, if given any $\A\in \B^{1}(\X,\Y)$ with $\T\perp_B \A$ implies that there exists $x_0\in S_{\X}$ such that $\|\T x_0\|=\|\T\|$ and $\T x_0\perp_B \A x_0.$ We conclude the article with the following corollary which shows that a smooth multilinear map also satisfies the Bhatia-\v{S}emrl property. 

\begin{cor}\label{Corollary:5}
Let $\X_i,~\Y$ be finite-dimensional Banach spaces; $1\leq i \leq k,$ and let the norm of $\B^{k}\left(\X_1\times \dots \times \X_k, \Y\right)$ is weakly differentiable at $\T$. Then given any $\A\in \B^{k}\left(\X_1\times \dots \times \X_k, \Y\right)$ with $\T\perp_B \A$ implies that there exists $\x_0\in M_\T$ (i.e., $\|\T\x_0\|=\|\T\|$) such that $\T \x_0\perp_B \A \x_0.$
\end{cor}

\begin{proof}
Let $\A\in \B^{k}\left(\X_1\times \dots \times \X_k, \Y\right)$ be arbitrary with $\T\perp_B \A.$ It follows from (i)$\implies$(ii) of Theorem \ref{Theorem:5} that
\[M_\T:=\left\{ \left(\mu^{(1)} x_1,\dots, \mu^{(k)} x_k\right)\in S_{\X_1}\times \dots \times S_{\X_k}:~\left(\mu^{(1)},    \dots, \mu^{(k)}\right)\in \mathbb{S}^k\right\},\]
and the norm of $\Y$ is weakly differentiable at $\T \x_0$, where $\x_0=(x_1,\dots,x_k)$. Then again from (ii)$\implies$(iii) of Theorem \ref{Theorem:5}, we have 
\[\Omega'(\T,\A)=\{y_0^*(\A \x_0)\},\]
where $y_0^*\in J(\T\x_0)$. Since $\T\perp_B \A$, applying Corollary \ref{Corollary:1}, we have $y_0^*(\A\x_0)=0$. Hence, $\T \x_0\perp_B \A \x_0$ and the proof is complete.
\end{proof}

\vspace{5mm}

\begin{acknowledgement}
 The author expresses sincere gratitude to Dr. Debmalya Sain for simulating conversations during the course of the work. The author is indebted to Professor Satya Bagchi for his guidance and constant encouragement. 
\end{acknowledgement}


\begin{thebibliography}{100}


\bibitem{Abatzoglou} T. J. Abatzoglou, \textit{Norm derivatives on spaces of operators}, \texttt{Math. Ann.}, \textbf{239} (1979), 129-135.

\medskip


\bibitem{Birkhoff} G. Birkhoff, \textit{Orthogonality in linear metric spaces}, \texttt{Duke Math. J.}, \textbf{1}, (1935), 169-172.


\medskip


\bibitem{BG} T. Bhattacharyya, P. Grover, \textit{Characterization of Birkhoff-James orthogonality}, \texttt{J. Math. Anal. Appl.}, \textbf{407} (2013), 350-358.


\medskip


\bibitem{BS} R. Bhatia and P. \v{S}emrl, \textit{Orthogonality of matrices and some distance problems}, \texttt{Linear Algebra Appl.}, \textbf{287} (1999), 77-85.


\medskip


\bibitem{Choi} G. Choi, S. K. Kim, \textit{The Birkhoff-James orthogonality and norm attainment for multilinear maps}, \texttt{J. Math. Anal. Appl.}, \textbf{502} (2021), 12575.


\medskip


\bibitem{Gustafson} K. Gustafson, \textit{The Toeplitz-Hausdorff Theorem for Linear Operators}, \texttt{Proc. Amer. Math. Soc.}, \textbf{25} (1970), 203-204.	


\medskip


\bibitem{Giles} J. R. Giles, \textit{Classes of semi-inner-product spaces}, \texttt{Trans. Amer. Math. Soc.}, \textbf{129} (1967) 436-446.


\medskip



\bibitem{Heinrich} S. Heinrich, \textit{The  differentiability of the  norm in spaces of operators}, \texttt{Functional. Anal. Appl.}, \textbf{9} (1975), 93-94.


\medskip


\bibitem{Hennefeld} J. Hennefeld, \textit{Smooth, compact operators}, \texttt{Proc. Amer. Math. Soc.}, \textbf{77} (1) (1979), 87-90.


\medskip


\bibitem{Holub} J. R. Holub, \textit{On the metric geometry of ideals of operators on Hilbert space}, \texttt{Math. Ann.}, \textbf{201} (1973), 157-163.




\medskip


\bibitem{James2} R. C. James, \textit{Orthogonality and linear functionals in normed linear spaces}, \texttt{Trans. Amer. Math. Soc.}, \textbf{61} (1947), 265-292.




\bibitem{Keckic} D. J. Ke\v{c}ki\`{c}, \textit{Gateaux derivative of B(H) norm},  \texttt{Proc. Amer. Math. Soc.}, \textbf{133} (2005), 2061-2067.


\medskip


\bibitem{Lumer} G. Lumer, \textit{Semi-inner-product spaces}, \texttt{Trans. Amer. Math. Soc.}, \textbf{100} (1961), 29-43.



\medskip






\bibitem{PHD} K. Paul, S.M. Hossein, K.C. Das, \textit{Orthogonality on B(H,H) and Minimal-Norm Operator}, J. Anal. Appl., \textbf{6} (2008), 169-178. 



\medskip



\bibitem{PSG} K. Paul, D Sain, P. Ghosh, \textit{Birkhoff–James orthogonality and smoothness of bounded linear operators},  \texttt{Linear Algebra Appl.}, \textbf{506} (2016), 551-563.



\medskip



\bibitem{PSMM} K. Paul, D. Sain, A. Mal, K. Mandal, \textit{Orthogonality of bounded linear operators on complex Banach spaces}, \texttt{Adv. Oper. Theory.}, \textbf{3} (2018), 699-709.



\medskip



\bibitem{Roy & Bagchi} S. Roy, S. Bagchi, \textit{Orthogonality of sesquilinear forms and spaces of operators}, \texttt{Linear Multilinear Algebra},\\
Available at - DOI: \texttt{https://doi.org/10.1080/03081087.2021.1881034}





\medskip




\bibitem{RSS} S. Roy, T. Senapati, D. Sain, \textit{Orthogonality of bilinear forms and application to matrices}, \texttt{Linear Algebra Appl.}, \textbf{615} (2021), 104-111.



\medskip



\bibitem{Sain} D. Sain, \textit{Birkhoff-James orthogonality of linear operators on finite dimensional Banach spaces}, \texttt{J. Math. Anal. Appl.}, \textbf{447} (2017),  860-866.


\medskip


\bibitem{Sain2} D. Sain, \textit{Smoothness and norm attainment of bounded bilinear operators between Banach spaces}, \texttt{Linear Multilinear Algebra},\\
Available at - DOI: \texttt{https://doi.org/10.1080/03081087.2019.1702917}



\medskip



\bibitem{Sain3} D. Sain, \textit{On the norm attainment set of a bounded linear operator and semi-Inner-Products in Normed Spaces}, \texttt{Indian J. Pure Appl. Math.}, \textbf{51} (2020), 179-186.


\medskip



\bibitem{Stampfli} J. G. Stampfli, \textit{The norm of a derivation}, \texttt{Pacific J. Math.}, \textbf{33} (1970), 737-747.



\medskip


\bibitem{Sain Paul} D. Sain, K. Paul, \textit{Operator norm attainment and inner product spaces}, \texttt{Linear Algebra Appl.}, \textbf{439} (2013), 2448-2452.


\medskip



\bibitem{SPH} D. Sain, K. Paul, S. Hait, \textit{Operator norm attainment and Birkhoff–James orthogonality}, \texttt{Linear Algebra Appl.}, \textbf{476} (2015), 85-97.


\medskip


\bibitem{SPM} D. Sain, K. Paul, A. Mal, \textit{A complete characterization of Birkhoff-James orthogonality in infinite dimensional normed space}, \texttt{J. Operator Theory}, \textbf{80} (2) (2018), 399-413.



\medskip



\bibitem{SPMR} D. Sain, K. Paul, A. Mal, A. Ray, \textit{A complete characterization of smoothness in the space of bounded linear operators}, \texttt{Linear Multilinear Algebra}, \textbf{68} (2020), 2484-2494.


\medskip


\bibitem{Turnsek} A. Turn\v{s}ek, \textit{A remark on orthogonality and symmetry of operators in B(H)}, \texttt{Linear Algebra Appl.}, \textbf{535} (2017), 141-150.
\end{thebibliography}
\end{document}